\documentclass{amsart}

\usepackage{calc} 

\usepackage[T1]{fontenc}
\usepackage{lmodern} 
\usepackage{lipsum}
\usepackage{enumitem}
\usepackage{amsmath,amsthm,amssymb}

\usepackage[
colorlinks=true,
linkcolor=blue
]{hyperref}
\usepackage{algorithm}
\usepackage{algpseudocode}

\usepackage{tikz}

\setlist[enumerate,1]{label=\textnormal{(\alph*)},ref=\textnormal{(\alph*)}}
\setlist[enumerate,2]{label=\textnormal{(\arabic*)},ref=\textnormal{(\arabic*)}}
\newcommand{\lref}[1]{\ref{#1}} 
\newcommand{\tlref}[1]{~\ref{#1}} 

\algnewcommand\algorithmicinput{\textbf{Input:}}
\algnewcommand\Input{\item[\algorithmicinput]}
\algnewcommand\algorithmicoutput{\textbf{Output:}}
\algnewcommand\Output{\item[\algorithmicoutput]}
\algnewcommand\algorithmicbreak{\textbf{break}}
\algnewcommand\Break{\algorithmicbreak}
\algnewcommand\algorithmiccontinue{\textbf{continue}}
\algnewcommand\Continue{\algorithmiccontinue}
\algnewcommand\algorithmicset{\textbf{set}}
\algnewcommand\Set[1]{#1}
\algnewcommand\algorithmicgoto{\textbf{goto}}
\algnewcommand\Goto[0]{\algorithmicgoto{} }

\newcommand{\dprobfont}[1]{\textsc{#1}}
\newcommand{\dprobfontinp}[1]{\textsc{#1}}
\newlength{\lenl}
\newlength{\lenm}
\newlength{\lenr}
\newcommand{\vproblem}[5]{
\setlength{\lenm}{\maxof{\widthof{#2\quad}}{\widthof{#4\quad}}}
\setlength{\lenr}{\textwidth-\lenl-\lenm}
\begin{flalign*}\arraycolsep=0pt
\begin{tabular}{@{}p{\lenl}@{}p{\lenm}@{}p{\lenr}@{}}
& \multicolumn{2}{@{}l@{}}{#1} \\
& \begin{minipage}[t]{\lenm}#2\vspace{1.5pt}\end{minipage} 
& \begin{minipage}[t]{\lenr}#3\vspace{1.5pt}\end{minipage} \\ 
& \begin{minipage}[t]{\lenm}#4\end{minipage} 
& \begin{minipage}[t]{\lenr}#5\end{minipage} \\
\end{tabular}&&
\end{flalign*}
}
\newcommand{\dproblem}[3]{\vproblem{\dprobfontinp{#1}}{Input:}{#2}{Problem:}{#3}}

\def\clap#1{\hbox to 0pt{\hss#1\hss}}

\def\mathclap{\mathpalette\mathclapinternal}

\def\mathclapinternal#1#2{%
\clap{$\mathsurround=0pt#1{#2}$}}

\newcommand{\rmsgg}[5]{\mathcal{M}{#5}(#1,#2,#3,#4)}
\newcommand{\rmsg}[4]{\rmsgg{#1}{#2}{#3}{#4}{}}
\newcommand{\rmsgz}[4]{\rmsgg{#1}{#2}{#3}{#4}{^0}}

\newcommand{\n}{\underline 1}
\newcommand{\1}{1}
\newcommand{\0}{0}

\newcommand{\Tsg}{T}
\newcommand{\Isg}{I}

\newcommand{\ptime}{\textnormal{P}}
\newcommand{\np}{\textnormal{NP}}
\newcommand{\pspace}{\textnormal{PSPACE}}

\newcommand{\exptime}{\textnormal{EXPTIME}}

\newcommand{\smp}{\dprobfont{SMP}}
\newcommand{\sat}{\dprobfont{SAT}}

\newcommand{\true}{true}
\newcommand{\false}{false}

\newcommand{\N}{\mathbb{N}}

\newcommand{\alg}[1]{#1}

\newcommand{\subuni}[1]{ {\langle #1 \rangle} }

\newcommand{\sst}{\mid}
\newcommand{\rg}[1]{[#1]}
\newcommand{\gj}{\mathrel{\mathcal{J}}}
\newcommand{\gh}{\mathrel{\mathcal{H}}}
\newcommand{\gr}{\mathrel{\mathcal{R}}}
\newcommand{\gl}{\mathrel{\mathcal{L}}}
\newcommand{\gd}{\mathrel{\mathcal{D}}}

\newcommand{\gjle}{\leq_{\mathcal{J}}}

\newcommand{\gjless}{<_{\mathcal{J}}}

\newcommand{\grle}{\leq_{\mathcal{R}}}

\newcommand{\glle}{\leq_{\mathcal{L}}}

\newtheorem{thm}{Theorem}[section]

\newtheorem{lma}[thm]{Lemma}
\newtheorem{clr}[thm]{Corollary}
\newtheorem{clm}[thm]{Claim}

\theoremstyle{definition}

\newtheorem{prb}[thm]{Problem}

\theoremstyle{remark}

\begin{document}
\title[On semigroups with PSPACE-complete SMP]{%
On semigroups with PSPACE-complete subpower membership problem}
\author{Markus Steindl}
\address{Institute for Algebra, Johannes Kepler University Linz, Altenberger St 69, 4040 Linz, Austria}
\address{Department of Mathematics, University of Colorado Boulder, Campus Box 395, Boulder, Colorado 80309-0395}
\email{\href{mailto:markus.steindl@colorado.edu}{markus.steindl@colorado.edu}}
\thanks{Supported by the Austrian Science Fund (FWF): P24285}
\date{\today}
\begin{abstract}
Fix a finite semigroup $S$ and let $a_1, \ldots, a_k, b$ be tuples in a direct power $S^n$.
The subpower membership problem (\smp) for $S$ asks whether $b$ can be generated by $a_1, \ldots, a_k$. 
For combinatorial Rees matrix semigroups we establish a dichotomy result:
if the corresponding matrix is of a certain form,
then the \smp{} is in \ptime; otherwise it is \np-complete.
For combinatorial Rees matrix semigroups with adjoined identity,
we obtain a trichotomy: the \smp{} is either in \ptime, \np-complete, or \pspace-complete.
This result yields various semigroups with \pspace-complete \smp{} including the $6$-element Brandt monoid, the full transformation semigroup on $3$ or more letters, 
and semigroups of all $n$ by $n$ matrices over a field for $n\ge 2$.
\end{abstract}
\maketitle

\section{Introduction}

In this paper we continue the investigation of the subpower membership problem (\smp) for semigroups
started in \cite{BMS2015} and \cite{smpbands}. 
At the Conference on Order, Algebra, and Logics in Nashville, 2007, Ross Willard proposed the \smp{} as follows \cite{Willard2007}:
Fix a finite algebraic structure $\alg S$ with
finitely many basic operations. Then the \emph{subpower membership problem} for $\alg S$
is the following decision problem:
\dproblem{SMP(\textit{S})}{
$\{a_1,\ldots,a_k\} \subseteq S^n,b \in S^n$}{
Is $b$ in the subalgebra of $S^n$ generated by $a_1,\ldots,a_k$?}
The \smp{} occurs in connection with the constraint satisfaction problem (CSP) \cite{IMMVW2010}.
In the algebraic approach to the CSP, each constraint relation is considered to be a subalgebra of a power (\emph{subpower}) of a certain finite algebra whose operations are the polymorphisms of the constraint language.
Checking whether a given tuple belongs to a constraint relation represented by its generators
is precisely the \smp{} for the polymorphism algebra. 

The input size of $\smp(S)$ is essentially $(k+1)n$.
We can always decide the problem using a straightforward closure algorithm in time exponential in $n$.
Thus $\smp(S)$ is in \exptime{} for every algebra $\alg S$.
However, the following questions arise:
\begin{itemize}
\item How does the algebra $S$ affect the computational complexity of $\smp(S)$?
\item For which algebras $S$ can $\smp(S)$ be solved in time polynomial in $k$ and $n$?
\item When is the problem complete in \np, \pspace{}, or \exptime?
Can it also be complete in a class other than these?
\end{itemize}

\noindent
Mayr \cite{Mayr2012} proved that the \smp{} for Mal'cev algebras is in~\np.
He also showed that for certain generalizations of groups and quasigroups the \smp{} is in \ptime{}.
Kozik \cite{Kozik2008} constructed a finite algebra with \exptime-complete \smp{}.

For semigroups the \smp{} is in \pspace.
This was shown in \cite{BMS2015} by Bulatov, Mayr, and the author of the present paper.
We also proved that the \smp{} of the full transformation semigroup on five letters is \pspace-complete. 
It was the first algebra known to have a \pspace-complete \smp{}.
In the same paper a dichotomy result for commutative semigroups was established: if a commutative semigroup $S$
embeds into a direct product of a Clifford semigroup and a nilpotent semigroup,
then $\smp(S)$ is in \ptime; otherwise it is \np-complete.

Another dichotomy for idempotent semigroups was established in \cite{smpbands}: if an idempotent semigroup $S$ satisfies a certain pair of quasiidentities, then $\smp(S)$ is in \ptime; otherwise it is \np-complete.

The first result of the current work is a condition for a semigroup $S$ under which $\smp(S)$ is \np-hard:

\begin{thm}\label{thm:nphard_condition}
Let $r,s,t$ be elements of a finite semigroup $S$ such that $s$ does not generate a group and $rs=st=s$.
Then $\smp(S)$ is $\np$-hard.
\end{thm}

\noindent
We will prove this result in Section~\ref{sec:nphard} by reducing the Boolean satisfiability problem SAT to $\smp(S)$.

A semigroup is called \emph{combinatorial} if every subgroup has one element.
Combinatorial \emph{Rees matrix semigroups} are of the form $\rmsgz{\{1\}}I\Lambda P$ (see \cite[Theorem 3.2.3]{Howie1995}). 
We give the following alternative notation:
For nonempty sets $I,\Lambda$ and a matrix $P \in \{0,1\}^{\Lambda \times I}$ we
let $S_P := (I \times \Lambda)\cup\{0\}$ and define a multiplication on $S_P$ by
\begin{gather*} 
[i,\lambda]\cdot[j,\mu] := 
\begin{cases} [i,\mu] &\text{if } P(\lambda,j)=1, \\ 
0 &\text{if } P(\lambda,j)=0,
\end{cases} \\
0\cdot[i,\lambda] := [i,\lambda]\cdot 0 := 0\cdot0 := 0.
\end{gather*}
It is easy to see that $S_P$ is indeed a combinatorial semigroup.
We say the matrix $P \in \{0,1\}^{\Lambda \times I}$ has \emph{one block} if there exist $J \subseteq I$, $\Delta \subseteq \Lambda$ such that 
for $i\in I$, $\lambda \in \Lambda$,
\[P(\lambda,i)=1 \quad\text{if and only if}\quad (\lambda,i)\in\Delta\times J. \]
For $P=\left(\begin{smallmatrix} 1 & 0 \\ 0 & 1	\end{smallmatrix}\right)$ we call $B_2 := S_P$ the \emph{Brandt semigroup}, 
and for $P=\left(\begin{smallmatrix} 1 & 1 \\ 1 & 0	\end{smallmatrix}\right)$ we denote $S_P$ by $A_2$.

In Section~\ref{sec:comb_reesmatrix} we establish the following two results:

\begin{thm}\label{thm:comb_rmsg_dichotomy}
Let $S_P$ be a finite combinatorial Rees matrix semigroup.
If the matrix $P$ has one block, 
then $\smp(S_P)$ is in \ptime. Otherwise $\smp(S_P)$ is \np-complete.
\end{thm}

\begin{clr}\label{clr:comb_reesmatrix_a2_b2}
The \smp{} for the Brandt semigroup $B_2$ and for the semigroup $A_2$ is \np-complete.
\end{clr}

\noindent
In Section~\ref{sec:pspace} we state a condition for semigroups $S$ under which $\smp(S)$ is \pspace-complete:

\begin{thm}\label{thm:mthm_pspace}
Let $S$ be a finite semigroup and $s,t,\n\in S$ such that
\begin{enumerate}
\item $sts = s$, 
\item $s$ does not generate a group,
\item $s\n = s$ and\/ $t\n = t$. 
\end{enumerate}
Then $\smp(S)$ is \pspace-complete.
\end{thm}

\noindent
In the proof we will reduce quantified \dprobfont{3SAT} to $\smp(S)$.
It follows that adjoining an identity to $B_2$ or $A_2$ already results in a \pspace-complete \smp:
\begin{thm}\label{thm:smp_a21_b21}
The \smp{} for the Brandt monoid $B_2^1$ and for the monoid $A_2^1$ is \pspace-complete.
\end{thm}

\noindent
This result is part of Corollary~\ref{clr1_mthm}. 
Both $B_2^1$ and $A_2^1$ embed into $\Tsg_3$, the full transformation semigroup on three letters. 
Thus $\smp(\Tsg_3)$ is also \pspace-complete.
So Theorem~\ref{thm:smp_a21_b21} generalizes the result from \cite{BMS2015} that 
$\smp(\Tsg_5)$
is \pspace-complete.
In addition, $B_2$ and $A_2$ are the first groupoids known to have an \np-complete \smp{} where adjoining an identity yields a groupoid with \pspace-complete \smp.
Further examples of semigroups with \pspace-complete \smp{} are listed in Section~\ref{sec:pspace}.

In Section~\ref{sec:rmsg_id} we will consider Rees matrix semigroups with adjoined identity
and prove the following trichotomy result:
\begin{thm}\label{thm:trichotomy_comb_rmsg_id}
Let $S_P$ be a finite combinatorial Rees matrix semigroup.
\begin{enumerate}
\item\label{it1_thm:trichotomy_comb_rmsg_id}
If all entries of the matrix $P$ are $1$, then $\smp(S_P^1)$ is in \ptime.
\item\label{it2_thm:trichotomy_comb_rmsg_id}
If $P$ has one block and some entries are $0$, then $\smp(S_P^1)$ is \np-complete.
\item\label{it3_thm:trichotomy_comb_rmsg_id}
Otherwise $\smp(S_P^1)$ is \pspace-complete.
\end{enumerate}
\end{thm}

\section{Semigroups with NP-hard SMP}
\label{sec:nphard}

In this section we will prove Theorem~\ref{thm:nphard_condition} by reducing the Boolean satisfiability problem \sat{} to $\smp(S)$.
It follows that the \smp{} for a semigroup $S$ is already \np-hard
if $S$ has a $\gd$-class that contains both group and non-group $\gh$-classes.

We denote Green's equivalences by $\gl,\gr,\gj,\gh,\gd$ \cite[p.~45]{Howie1995}. For the definition of the related preorders $\glle,\grle,\gjle$ see \cite[p.~47]{Howie1995}.
We write $\rg{n}:=\{1,\dots,n\}$ for $n\in\N$ and set $\rg{0} := \varnothing$.
We consider a tuple $a$ in a direct power $S^n$ to be a function $a\colon \rg{n} \rightarrow S$. 
This means the $i$th coordinate of this tuple is denoted by $a(i)$ rather than $a_i$.
The subsemigroup generated by a set $A = \{ a_1,\ldots,a_k \}$ may be denoted by $\subuni{A}$ or $\subuni{a_1,\ldots,a_k}$.

\begin{lma}\label{lma:j_h_classes_1}
Let $s$ belong to a finite semigroup $S$.
Then $s$ generates a group if and only if $s^2\gj s$.
\end{lma}

\begin{proof}
If $s$ generates a group, then $s^k=s$ for some $k\ge2$. Thus $s^2\gj s$.

For the converse let $s^2 \gj s$.
First assume the $\gj$-class $J_s$ is the minimal ideal of $S$. 
Then $J_s$ is a finite simple semigroup by \cite[Proposition 3.1.4]{Howie1995}. 
By the Rees Theorem for finite simple semigroups $s$ generates a group. 

Now assume $J_s$ is not the minimal ideal of $S$.
Let
\begin{align*}
J(s) &:= \{r\in S \mid r \gjle s \}, 
\quad I(s) := \{r\in S \mid r \gjless s \}.
\end{align*}
By \cite[Proposition 3.1.4]{Howie1995} the \emph{principal factor} ${J(s)}/{I(s)}$ is either null or $0$-simple. 
Since $s^2 \gj s$, the second case applies.
By the Rees Theorem for finite $0$-simple semigroups $s$ generates a group.
\end{proof}

\begin{lma}\label{lma:hardness_jclass_argument}
Let $r,s,t$ be elements of a finite semigroup $S$ such that $s$ does not generate a group and $rs = st = s$. Then there are idempotents $e,f \in S$ such that $es = sf = s$ and every product $a_1 \cdots a_k$ in $s,e,f$ in which $s$ occurs at least twice does not yield $s$.
\end{lma}

\begin{proof}
First assume $s$ is regular, i.e.\ $sus=s$ for some $u\in S$.
Let $e$ and $f$ be the idempotent powers of $su$ and $us$ respectively.
Clearly $es=sf=s$.
Let $a_1 \cdots a_k$ be a product in $s,e,f$, and $i < j$ such that $a_i=a_j=s$.
Let $\ell \in \{i+1,\ldots,j\}$ be maximal such that $a_{i+1}=\ldots=a_{\ell-1}=f$.
Then $a_i \cdots a_{\ell-1} = s$, and thus $a_i \cdots a_\ell \in \{s^2,se\}$.
Note that $se=s(su)^m$ for some $m\in\N$. 
Now a factor $s^2$ occurs in the product $a_i \cdots a_\ell$.
Since $s$ does not generate a group, Lemma~\ref{lma:j_h_classes_1} implies that $s^2\gjless s$. 
Thus $a_1 \cdots a_\ell \gjless s$, and the result follows.

Now assume $s$ is not regular. 
By \cite[Theorem 3.1.6]{Howie1995} the principal factor $J(s)/I(s)$ is null.
Let $e$ and $f$ be the idempotent powers of $r$ and $t$ respectively.
Let $a_1 \cdots a_k$ be a product in $s,e,f$, and let $i < j$ such that $a_i=a_j=s$.
Then $a_1\cdots a_i \gjle s$ and $a_{i+1}\cdots a_k \gjle s$. Since $J(s)/I(s)$ is null, it follows that $a_1\cdots a_k \gjless s$.
\end{proof}

\begin{proof}[Proof of Theorem~\ref{thm:nphard_condition}]
Let $S$ satisfy the assumptions.
We reduce the Boolean satisfiability problem \sat{} to $\smp(S)$.
\sat{} is \np-complete \cite{Cook1971}, and we give the following definition:
\dproblem{\sat}{Clauses $C_1, \ldots, C_m \subseteq \{x_1, \ldots, x_k, \neg x_1, \ldots, \neg x_k\}$.}
{Do truth values for $x_1, \ldots, x_k$ exist for which the Boolean formula 
$\Phi( x_1, \ldots, x_k ) := (\bigvee C_1) \wedge\ldots\wedge (\bigvee C_m)$ 
is true?}

\noindent
For all $j\in\rg{k}$ we may assume that $x_j$ or $\neg x_j$ occurs in some clause $C_i$.
We define an $\smp(S)$ instance
\begin{equation*}\label{eq_10_thm:nphard_condition} 
A := \{a_1^0,\ldots,a_k^0,a_1^1,\ldots,a_k^1\} \subseteq S^{k+m},\ b \in S^{k+m}.
\end{equation*}
Let $e,f\in S$ be idempotents with the properties from Lemma~\ref{lma:hardness_jclass_argument}.
Let $g$ be the idempotent power of $se$. 
Observe that $e$ and $g$ form a two-element semilattice with $g < e$.

Let $j \in \rg{k}$ and $z \in \{ 0, 1 \}$. 
For $i\in\rg{k}$ let
\begin{equation*}
\phantom{a_j^{1}(k+i)}%
\llap{$a_j^{z}(i)$} := 
\rlap{$\begin{cases}
f & \text{if } i < j, \\
s & \text{if } i = j, \\
e & \text{if } i > j,
\end{cases}$}
\phantom{\begin{cases}
g & \text{if } \neg x_j \in C_i, \\
e & \text{otherwise,}
\end{cases}}
\end{equation*}
and for $i\in\rg{m}$ let
\begin{align*}
a_j^0(k+i) &:= 
\begin{cases}
g & \text{if } \neg x_j \in C_i, \\
e & \text{otherwise,}
\end{cases}
\\
a_j^1(k+i) &:= 
\begin{cases}
g & \text{if } x_j \in C_i, \\
e & \text{otherwise.}
\end{cases}
\end{align*}
Let
\begin{align*}\begin{aligned}
b(i)  &:= s & \text{for } & i\in\rg{k}, \\
b(k+i) &:= g & \text{for } & i\in\rg{m}.
\end{aligned}\end{align*}
We claim that
\begin{equation}\label{eq2_thm:nphard_condition}
\text{the Boolean formula $\Phi$ is satisfiable if and only if $b\in\subuni A$.}
\end{equation}

\noindent
For the $(\Rightarrow)$ direction let $z_1, \ldots, z_k \in \{ 0,1 \}$ 
such that $\Phi(z_1,\ldots,z_k)=1$.
We show that 
\begin{equation}\label{eq1_thm:nphard_condition}
b=a_1^{z_1} \cdots a_k^{z_k}.
\end{equation}
For $i\in\rg{k}$ we have $a_1^{z_1} \cdots a_k^{z_k} (i) = {e}^{i-1} s {f}^{k-i} = s = b(i)$.
For $i\in\rg{m}$ the clause $\bigvee C_i$ is satisfied under the assignment $x_1\mapsto z_1,\ldots,x_k\mapsto z_k$.
Thus there is a $j\in\rg{k}$ such that $x_j\in C_i$ and $z_j = 1$, or $\neg x_j \in C_i$ and $z_j = 0$.
In both cases $a_j^{z_j}(k+i) = g$, and thus $a_1^{z_1} \cdots a_k^{z_k}(k+i) = g = b(k+i)$.
This proves \eqref{eq1_thm:nphard_condition} and the $(\Rightarrow)$ direction of \eqref{eq2_thm:nphard_condition}.

For the $(\Leftarrow)$ direction of \eqref{eq2_thm:nphard_condition} assume $b = a_{j_1}^{z_1} \cdots a_{j_\ell}^{z_\ell}$ for some $\ell\in\N$, $j_1,\ldots,j_\ell \in \rg{k}$, and $z_1,\ldots,z_\ell\in\{0,1\}$.
We show that $j_1, \ldots, j_\ell$ are distinct. 
Suppose $j_p$ = $j_q$ for $p < q$. 
The factors of the product $a_{j_1}^{z_1}\cdots a_{j_\ell}^{z_\ell}(j_p)$ are given by $s,e,f$. The factor $s$ occurs at least twice since $a_{j_p}(j_p)=a_{j_q}(j_p)=s$.
By Lemma~\ref{lma:hardness_jclass_argument} this product does not yield $s$, contradicting our assumption.
We define an assignment 
\begin{align*}
\theta \colon x_{j_1}&\mapsto z_1,\ldots,x_{j_\ell}\mapsto z_\ell, \\
x_j &\mapsto 0 \quad\text{for } j \in \rg k \setminus \{ j_1, \ldots, j_\ell \},
\end{align*}
and show that $\theta$ satisfies the formula $\Phi$.
Let $i \in \rg m$. 
Since $a_{j_1}^{z_1} \ldots a_{j_\ell}^{z_\ell}(k+i)$ 
is a product in $e,g$ that yields $g$, some factor $a_{j_p}^{z_p}(k+i)$ must be $g$. 
From the definition of $a_{j_p}^{z_p}$ we see that either
$z_p = 0$ and $\neg x_{j_p}\in C_i$, or $z_p = 1$ and $x_{j_p} \in C_i$.
This means the formula $\bigvee C_i$ is satisfied under the assignment $\theta$.
Since $i$ was arbitrary, $\Phi$ is also satisfied.
The equivalence \eqref{eq2_thm:nphard_condition} and the theorem are proved.
\end{proof}

\begin{clr}
\label{clr:hardness_jclass_argument}
If a $\gj$-class of a finite semigroup $S$ contains both group and non-group $\gh$-classes, then $\smp(S)$ is \np-hard.
\end{clr}
\begin{proof}
Let $s\in S$ such that $H_s$ is not a group and $J_s$ contains group $\gh$-classes.
From Green's Theorem \cite[Theorem 2.2.5]{Howie1995} we know that $s$ does not generate a group.
Since $J_s$ contains an idempotent and $J_s=D_s$, the element $s$ is regular by~
\cite[Proposition 2.3.1]{Howie1995}. 
That is, there is a $u\in S$ such that $sus=s$. 
Now $su$, $s$, and $us$ fulfill the hypothesis of Theorem~\ref{thm:nphard_condition}.
\end{proof}

\section{Combinatorial Rees matrix semigroups}
\label{sec:comb_reesmatrix}

In this section we will establish
a \ptime/\np-complete dichotomy for the \smp{} for combinatorial Rees matrix semigroups by proving Theorem~\ref{thm:comb_rmsg_dichotomy}.
After that we apply this result to combinatorial $0$-simple semigroups.

Combinatorial Rees matrix semigroups have the following property:

\begin{lma}[cf. {\cite[Lemma 2.2]{SS2006}}]\label{lma:combrmsg1}
Let $k\ge 2$ and $a_1,\ldots,a_k$ be elements of a combinatorial Rees matrix semigroup $S_P$. 
\begin{enumerate}
\item\label{it1_lma:combrmsg1}
We have $a_1 \cdots a_k = 0$ if and only if $a_j a_{j+1}=0$ for some $j\in\rg{k-1}$.
\item\label{it2_lma:combrmsg1}
If $a_1 \cdots a_k \neq 0$, then there are $i,j \in I$ and $\lambda,\mu \in \Lambda$ such that $a_1=[i,\lambda]$, $a_k=[j,\mu]$, and $a_1 \cdots a_k = [i,\mu]$. 
\end{enumerate}
\end{lma}

\begin{proof}Straightforward.\end{proof}

The next two results will allow us to show that the \smp{} for a combinatorial Rees matrix semigroup is in \np.

\begin{lma}[cf. {\cite[Theorem 4.3]{SS2006}}]\label{lma:combrmsg2}
Let $f := y_1\cdots y_k$ and $g := z_1 \cdots z_\ell$ be words over an alphabet $X$ such that
\begin{enumerate}
\item\label{it1_lma:combrmsg2} 
$\{ y_iy_{i+1} \sst i\in\rg{k-1} \} = \{ z_jz_{j+1} \sst j\in\rg{\ell-1} \}$,
\item\label{it2_lma:combrmsg2} 
$y_1 = z_1$ and $y_k = z_\ell$.
\end{enumerate}
Then every combinatorial Rees matrix semigroup $S_P$ satisfies $f \approx g$.
\end{lma}

\begin{proof}
Let $S_P$ be a combinatorial Rees matrix semigroup, and let
$\alpha \colon X^+ \to S_P$ be a homomorphism from the free semigroup over $X$ to $S_P$. 
By item~\lref{it1_lma:combrmsg2} we have $\{y_1,\ldots,y_k\}=\{z_1,\ldots,z_\ell\}$.
We claim that 
\begin{equation}\label{eq1_lma:combrmsg2}
\alpha(y_1 \cdots y_k) = 0 \quad\text{if and only if}\quad \alpha(z_1 \cdots z_\ell)=0.
\end{equation}
Assume $\alpha(y_1 \cdots y_k)=0$.
Then $\alpha(y_i)\alpha(y_{i+1})=0$ for some $i\in\rg{k-1}$ by Lemma~\ref{lma:combrmsg1}\tlref{it1_lma:combrmsg1}.
By item \lref{it1_lma:combrmsg2} 
$y_iy_{i+1}=z_jz_{j+1}$ for some $j\in\rg{\ell-1}$. Thus $\alpha(z_j)\alpha(z_{j+1})=0$, and hence
$\alpha(z_1 \cdots z_\ell)=0$. This proves \eqref{eq1_lma:combrmsg2}.

If $\alpha(y_1 \cdots y_k) = 0$, then $\alpha(y_1 \cdots y_k) = \alpha(z_1 \cdots z_\ell)$ by \eqref{eq1_lma:combrmsg2}.
Assume $\alpha(y_1 \cdots y_k) \neq 0$. Then also $\alpha(z_1 \cdots z_\ell) \neq 0$, and 
Lemma~\ref{lma:combrmsg1}\tlref{it2_lma:combrmsg1} implies
\[ \alpha(y_1)\cdots\alpha(y_k) = \alpha(z_1)\cdots\alpha(z_\ell). \]
This proves the lemma.
\end{proof}

\begin{lma}\label{lma:combrmsg3}
Let $f$ be a word over $x_1,\ldots,x_k$.
Then there is a word $g$ such that 
\begin{enumerate}
\item\label{it1_lma:combrmsg3} the length of $g$ is at most $k(k^2+1)$, and
\item\label{it2_lma:combrmsg3} every combinatorial Rees matrix semigroup satisfies $f \approx g$.
\end{enumerate}
\end{lma}

\begin{proof}
Let $f=y_1 \cdots y_\ell$ for $y_1,\ldots,y_\ell \in \{x_1,\ldots,x_k\}$. 
We show that there is a word $g$ 
such that item~\lref{it2_lma:combrmsg3} holds and
in which each variable $x_i$ occurs at most $k^2+1$ times.
Fix $i\in\rg{k}$. 
Let $j_1,\ldots,j_m \in \rg{\ell}$ be the positions of $x_i$ in $y_1\cdots y_\ell$.
Let
\begin{align*}\begin{aligned}
v_1 &:= y_1 \cdots y_{j_1}, && \\
v_r &:= y_{j_{r-1}+1}\cdots y_{j_r} &&\text{for } r \in \{2,\ldots,m\}, \\
v_{m+1} &:= y_{j_m+1}\cdots y_\ell. &&
\end{aligned}\end{align*}
Note that $f = v_1\cdots v_{m+1}$.
Now for every word $h := z_1\cdots z_n$ over $x_1,\ldots,x_k$ let
\[ E(h) := \{ z_jz_{j+1} \sst j \in \rg{n-1} \}. \]
It is not hard to see that 
\[\text{$E(v_1\cdots v_r)=E(v_1)\cup E(x_iv_2) \cup \ldots \cup E(x_iv_r)$ \quad{}for $r\in\{2,\ldots,m+1\}$}.\]
We define \[ R := \{ r\in\{2,\ldots,m\} \sst E(x_iv_r) \not \subseteq E(v_1\cdots v_{r-1})\} \]
and let
\[ g := v_1 ( \prod_{\mathclap{r \in R}} v_r ) v_{m+1}.\]
Apparently $g$ is a concatenation of subwords of $f$, and $f$ and $g$ start with the same letter.
We show that 
\begin{equation}\label{eq:blafoo}
\text{$f$ and $g$ also end with the same letter.}
\end{equation}
If $v_{m+1}$ is nonempty, then \eqref{eq:blafoo} is clear.
If $v_{m+1}$ is empty, then $y_\ell=x_i$, and $g$ ends with a subword $v_r$ for some $r\in\rg{m}$.
Since $v_r$ and $f$ both end with $x_i$, \eqref{eq:blafoo} is proved.
We have 
\begin{align*} 
E(f) = E(v_1\cdots v_{m+1}) 
&= E(v_1) \cup \bigcup_{\mathclap{r=2}}^m E(x_iv_r) \cup E(x_iv_{m+1}) \\
&= E(v_1) \cup \bigcup_{\mathclap{r\in R}} E(x_iv_r) \cup E(x_iv_{m+1}) = E(g).
\end{align*}
Now Lemma~\ref{lma:combrmsg2} implies item~\lref{it2_lma:combrmsg3}.

Next observe that $|R| \le k^2$ by the definitions of $R$ and $E$.
This means $x_i$ occurs at most $k^2+1$ times in $g$.
Since $x_i$ was arbitrary, we can reduce the number of occurrences of each variable in $f$ to at most $k^2+1$. Item~\lref{it1_lma:combrmsg3} is proved.
\end{proof}

\begin{lma}\label{lma:combrmsg3_np}
The \smp{} for a finite combinatorial Rees matrix semigroup is in \np{}.
\end{lma}

\begin{proof}
Let $S$ be such a semigroup, and let $\{a_1,\ldots,a_k\} \subseteq S^n,\, b\in S^n$ be an instance of $\smp(S)$.
If $b \in \subuni{a_1,\ldots,a_k}$, then there is a term function $f$ such that $f(a_1,\ldots,a_k)=b$. By Lemma~\ref{lma:combrmsg3} there is a word $g$ which induces $f$ and
whose length is polynomial in $k$. Now $g$ witnesses the positive answer.
\end{proof}

For the following result note that the all-$0$ matrix has one block.

\begin{lma}\label{lma:combrmsg_ptime}
Let $S_P$ be a finite combinatorial Rees matrix semigroup such that $P\in\{0,1\}^{\Lambda \times I}$ has one block.
Then Algorithm~\ref{alg:smp_combrmsg} decides $\smp(S_P)$ in polynomial time.
\end{lma}

\begin{algorithm}
\caption{Decides $\smp(S_P)$ in polynomial time if $P$ has one block.}
\label{alg:smp_combrmsg}
\begin{algorithmic}[1]
\Input 
\begin{minipage}[t]{\textwidth-\widthof{\textbf{Input: }}}
$A \subseteq {S_P}^n,\, b \in {S_P}^n$,  \\
$m\in\{0,\ldots,n\}$ such that $b(i)\ne 0$ iff $i\in\rg{m}$, \\
$J\subseteq I$, $\Delta \subseteq \Lambda$ such that $P(\lambda,i)=1$ iff $(\lambda,i) \in \Delta \times J$ for $i\in I$, $\lambda \in \Lambda$.
\end{minipage}
\Output{\true{} if $b \in \subuni{ A }$, \false{} otherwise. }
	\If{$b\in A$}
		\State\Return\true{}
		\label{ret0_alg:smp_combrmsg}
	\EndIf
	\State $d := \prod \{ a\in A \mid a(\rg{m})\subseteq J\times\Delta \}$\quad{}(some order)
	\label{set_d_alg:smp_combrmsg}
	\State\Return $\exists a_1,a_2 \in A \colon a_1da_2 = b$
	\label{ret1_alg:smp_combrmsg}
\end{algorithmic}
\end{algorithm}

\begin{proof}
Fix an input $A \subseteq {S_P}^n,\, b \in {S_P}^n$.
We may assume that there is an $m \in \{0,\ldots,n\}$ such that
\begin{align*}\begin{aligned}
b(i) &\neq 0 && \text{for } i \in \rg{m}, \\
b(i) &= 0 && \text{for } i \in \{m+1,\ldots,n\}.
\end{aligned}\end{align*}

\emph{Correctness of Algorithm~\ref{alg:smp_combrmsg}.} 
If Algorithm~\ref{alg:smp_combrmsg} returns \true{}, then clearly $b\in\subuni{A}$.
Conversely assume $b\in\subuni{A}$. We show that \true{} is returned.
Let $g_1,\ldots,g_k \in A$ such that $b=g_1 \cdots g_k$.
If $k=1$ then \true{} is returned in line~\ref{ret0_alg:smp_combrmsg}.
Assume $k \ge 2$.
We have
\begin{gather}\label{eq1_alg:smp_combrmsg}
\begin{split}
g_1(i) \in I \times \Delta, \quad g_k(i) \in J \times \Lambda, &\quad\text{and} \\
g_2(i),\ldots,g_{k-1}(i) \in J \times \Delta &\quad\text{for all } i\in\rg{m};
\end{split}
\end{gather}
otherwise we obtain the contradiction $g_1\cdots g_k(i)=0$ for some $i\in\rg{m}$.
Let $d$ have a value assigned by line~\ref{set_d_alg:smp_combrmsg}.
We claim that
\begin{equation}\label{eq1.5_alg:smp_combrmsg}
g_1 d g_k = b.
\end{equation}
For $i\in\rg{m}$ we have $d(i)\in J\times\Delta$.
The multiplication rule and \eqref{eq1_alg:smp_combrmsg} imply
\[ b(i) = g_1 \cdots g_k(i) = g_1dg_k(i). \]
Now let $i \in \{m+1,\ldots,n\}$.
Since $b(i)=0$, there are three cases:
$g_1(i) \notin I \times \Delta$, $g_k(i) \notin J \times \Lambda$, or $g_j(i) \notin J \times \Delta$ for some $j\in \{2,\ldots,k-1\}$.
In the first two cases $g_1dg_k(i)=0=b(i)$ holds.
In the third case $a:=g_j$ occurs as factor in line~\ref{set_d_alg:smp_combrmsg}.
Thus $d(i)\notin J\times\Delta$, and hence $g_1dg_k(i)=0$. This proves \eqref{eq1.5_alg:smp_combrmsg}. 
So the algorithm returns \true{} in line~\ref{ret1_alg:smp_combrmsg}.

\emph{Complexity of Algorithm~\ref{alg:smp_combrmsg}.} The product in line~\ref{set_d_alg:smp_combrmsg} can be computed in $\mathcal{O}(|A|n)$ time.
Checking the condition in line \ref{ret1_alg:smp_combrmsg} requires $\mathcal{O}(|A|^2n)$ time.
Altogether Algorithm~\ref{alg:smp_combrmsg} runs in $\mathcal{O}(|A|^2n)$ time.
\end{proof}

Now we prove Theorem~\ref{thm:comb_rmsg_dichotomy} and Corollary~\ref{clr:comb_reesmatrix_a2_b2}.

\begin{proof}[Proof of Theorem~\ref{thm:comb_rmsg_dichotomy}]
Assume $P\in\{0,1\}^{\Lambda \times I}$.
If $P$ has one block, then $\smp(S_P)$ is in \ptime{} by Lemma~\ref{lma:combrmsg_ptime}.
Assume $P$ does not have one block.
Then there are $i,j \in I$ and $\lambda,\mu \in \Lambda$ such that 
\[ P(\lambda,i) = P(\mu,j) = 1 \quad\text{and}\quad P(\mu,i)=0. \]
Let $r:=[i,\lambda]$, $s:=[i,\mu]$, and $t:=[j,\mu]$.
Then $rs=st=s$, and $s$ does not generate a group.
By Theorem~\ref{thm:nphard_condition} $\smp(S_P)$ is \np-hard.
\np-easiness follows from Lemma~\ref{lma:combrmsg3_np}.
\end{proof}

\begin{proof}[Proof of Corollary~\ref{clr:comb_reesmatrix_a2_b2}]
The result is immediate from Theorem~\ref{thm:comb_rmsg_dichotomy}.
\end{proof}

Next we restate the Rees Theorem (see \cite[Theorem 3.2.3]{Howie1995}) for the case of finite combinatorial $0$-simple semigroups:
\begin{thm}[Rees Theorem]\label{thm:rees_comb0simple}
Let $P$ be a finite $0$-$1$ matrix such that each row and each column has at least one 1.
Then $S_P$ is a finite combinatorial $0$-simple semigroup.

Conversely, every finite combinatorial\/ $0$-simple semigroup is isomorphic to one constructed in this way.
\end{thm}
\begin{proof} See \cite{Howie1995}. \end{proof}

\begin{lma}\label{lma:0simple_oneblock}
Let $S_P$ be a finite combinatorial\/ $0$-simple semigroup. 
Then the matrix $P$ has one block if and only if $S_P$ has no zero divisors, i.e.\ 
for $s,t\in S_P$, $st=0$ implies that $s=0$ or $t=0$.
\end{lma}

\begin{proof}
Assume $P\in\{0,1\}^{\Lambda \times I}$.
If $P$ has one block, then all entries of $P$ are $1$.
Thus $S_P$ has no zero divisors.
If $P$ does not have one block, then $P(\lambda,i)=0$ for some $\lambda \in \Lambda$, $i \in I$.
Now $[i,\lambda]$ is a zero divisor since $[i,\lambda]^2=0$.
\end{proof}

\begin{clr}\label{thm:comb_0simple_dichotomy}
If a finite combinatorial\/ $0$-simple semigroup $S$ has no zero divisors, then $\smp(S)$ is in \ptime.
Otherwise $\smp(S)$ is \np-complete.
\end{clr}

\begin{proof}
The result is immediate from Theorem~\ref{thm:comb_rmsg_dichotomy} and Lemma~\ref{lma:0simple_oneblock}.
\end{proof}

\section{Semigroups with \pspace-complete \smp}
\label{sec:pspace}

In \cite{BMS2015} an upper bound on the complexity of the \smp{} for semigroups was established:
\begin{thm}[{\cite[Theorem 2.1]{BMS2015}}]\label{thm:pspace}
The \smp{} for a finite semigroup is in \pspace.
\end{thm}

\begin{proof}
Let $S$ be a finite semigroup. We show that 
\begin{equation} \label{eq:NPS}
\smp(S) \text{ is in nondeterministic linear space.}
\end{equation}
To this end, let $A\subseteq S^n,\,b\in S^n$ be an instance of $\smp(S)$.
If $b\in\subuni{A}$, then there exist $a_1,\dots,a_m\in A$ such that $b = a_1\cdots a_m$.

Now we pick the first generator $a_1 \in A$ nondeterministically and start with $c := a_1$.
Pick the next generator $a \in A$ nondeterministically, compute $c := c\cdot a$, and repeat until 
we obtain $c = b$. Clearly all computations can be done in space linear in $|A|n$.
This proves~\eqref{eq:NPS}. By a result of Savitch~\cite{Savitch1970} this implies that $\smp(S)$ is in deterministic quadratic space.
\end{proof}

In \cite{BMS2015} it was shown that the \smp{} for the full transformation semigroup on $5$ letters is \pspace-complete by reducing \dprobfont{Q3SAT} to $\smp(\Tsg_5)$.
We adapt the proof of this result and show that under the following conditions 
the \smp{} for a semigroup is \pspace-complete.

\begin{lma}\label{lma:mthm_pspace}
Let $S$ be a finite semigroup and $s,t,\n \in S$ such that
\begin{enumerate}
\item\label{it1_lma:mthm_pspace}
$sts = s$, $tst = t$,
\item $s^2,t^2 \gjless s$,
\item $s\n = s$ and $t\n = t$. 
\end{enumerate}
Then $\smp(S)$ is \pspace-complete.
\end{lma}

\begin{proof}
From item~\lref{it1_lma:mthm_pspace} we know that $s,t,st,ts$ are in the same $\gj$-class.
Observe that $s\ne st$; otherwise $s^2=sts=s$, which is impossible.
We consider $s$ and $st$ as states and let $\n,s,t,st,ts$ act on these states by multiplication on the right.
This yields the partial multiplication table
\begin{equation}\label{eq05_mthm}
\begin{array}{c|ccccc}
S  & \n   & s    & t    & st   & ts    \\ \hline
s  & s    &\infty& st   &\infty& s     \\
st & st   & s    &\infty& st   &\infty \\
\end{array}
\end{equation}
where $\infty$ means that this entry is $\gjless s$.

$\smp(S)$ is in \pspace{} by Theorem~\ref{thm:pspace}.
For \pspace-hardness we reduce \dprobfont{Q3SAT} to $\smp(S)$.
\dprobfont{Q3SAT} is \pspace-complete~\cite{Pa:CC} and can be defined as follows.
\dproblem{Q3SAT}{
triples $C_1,\ldots,C_m$ over\\
$\{x_1,\ldots,x_n,\neg x_1,\ldots, \neg x_n,y_1,\ldots,y_n,\neg y_1,\ldots,\neg y_n\}$
}{%
Is the Boolean formula \\
$\Phi := \forall x_1 \exists y_1 \ldots \forall x_n \exists y_n \, (\bigvee C_1) \wedge \ldots \wedge (\bigvee C_m)$ true?}
Let $C_1,\ldots,C_m$ be a \dprobfont{Q3SAT} instance and $\Phi$ be the corresponding Boolean formula.
We refer to $x_1,\ldots,x_n$ as \emph{universal variables}, to $y_1,\ldots,y_n$ as \emph{existential variables}, and to $C_1,\ldots,C_m$ as \emph{clauses}.

We define the corresponding $\smp(S)$ instance
\begin{align*} 
G \subseteq S^{3n+m+1},\ f \in S^{3n+m+1}
\end{align*}
where
\begin{align*}
G &:= \{a\} \cup B\cup C\cup D\cup E, \\
B &:= \{b_1,\ldots,b_n\}, \\
C &:= \{ c_{j}^k \sst j \in \rg{m},\ k \in \{+,-,0\} \}, \\
D &:= \{ d_{jk} \sst j \in \rg{m},\ k \in \rg{3} \}, \\ 
E &:= \{e_1,\ldots,e_n\}.
\end{align*}
The coordinates of these tuples will have the following meaning. The first $n$ positions encode truth values assigned to $x_1,\ldots,x_n$, and the second $n$ positions truth values assigned to $y_1,\ldots,y_n$. 
Truth values are represented by
\begin{align*} \0 := s \quad\text{and}\quad \1 &:= st. \end{align*}

The positions $2n+1$ to $3n$ control the order in which the given tuples are multiplied.
The $m$ positions after that indicate the status of the clauses $C_1,\ldots,C_m$. In particular, for $i\in\rg{m}$ position $3n+i$ encodes whether $\bigvee C_i$ is satisfied by the assignment given by the first $2n$ positions. We encode 
\[ \text{``unsatisfied'' by } s \text{, and ``satisfied'' by } st.\]
The last position ensures that the first generator of the target tuple $f$ is the tuple $a$.
The generators are explicitly given as follows. 
For an overview see Figure~\ref{fig:gens}.

\begin{itemize}
\item
We define our target tuple $f$ by
\begin{align*}\begin{split}
f(i) &:= st \quad\text{for } i\in\rg{3n+m}, \\
f(3n+m+1) &:= s.
\end{split}\end{align*}

\item The tuple $a$ will be the first generator of $f$. It encodes the all-zero assignment for the variables $x_1,\ldots,x_n,y_1,\ldots,y_n$. Let
\begin{align*} 
a(i) := s \quad\text{for } i \in \rg{3n+m+1}.
\end{align*}
The idea is that $a$ is multiplied only on the right by elements of $G$.
The components of this product are the states $s$ and $st$. 
Each multiplication on the right modifies the states according to the multiplication table \eqref{eq05_mthm}.

\item For $j \in \rg{n}$ let $b_j$ change the assignment for the universal variables 
\[ \text{from $(x_1,\ldots,x_{j-1},\0,\1,\ldots,\1)$ to $(x_1,\ldots,x_{j-1},\1,\0,\ldots,\0)$.}\]

For $j\in\rg{n}$ define
\begin{flalign*}\begin{split}
b_j(i) &:= \begin{cases} 
\n &\text{if } i \in \rg{j-1} \text{ or } n+1 \le i \le 2n, \\
t &\text{if } i=j, \\ 
s & \text{if } j < i \le n, \\ 
 \end{cases} \\
b_j(2n+i) &:= \begin{cases} 
st &\text{if } i\in\rg{j-1}, \\ 
s &\text{if } j\le i \le m+n, \end{cases} \\
b_j(3n+m+1) &:= ts. \\
\end{split}\end{flalign*}

\item For $j\in\rg n$ let $c_j^+$ and $c_j^-$ change the assignment for the existential variable $y_j$ from $\0$ to $\1$ and from $\1$ to $\0$, respectively. Let $c_j^0$ leave the variables unchanged.
For $j \in \rg{n}$ and $i \in \rg{3n+m+1}$ let
\begin{align*}
\hphantom{c_j^-(n+j)}
\llap{$c_j^0(i)$} &:= \begin{cases} 
t &\text{if $i=2n+j$,} \\ 
st &\text{if $i=3n+m+1$,} \\ 
\n &\text{otherwise.} \end{cases}
\end{align*}
The tuples $c_j^+$ and $c_j^-$ differ from $c_j^0$ only in the following positions:
\begin{align*}
c_j^+(n+j) &:= t, \quad
c_j^-(n+j) := s.
\end{align*}

\item
For $j \in \rg{m}$ and $k \in \rg{3}$ the tuple $d_{jk}$ evaluates the $k$th literal $C_{jk}$ of the $j$th clause. If this literal is satisfied by the assignment encoded in the first $2n$ components, then multiplying by $d_{jk}$ changes the status of the clause $C_j$  to ``satisfied''. This will be more formally stated in Claim~\ref{cl1_mthm}.
For $i \in \rg n$ define
\begin{align*}\begin{split}
\hphantom{d_{jk}(3n+m+1)}
\llap{$d_{jk}(i)$} &:= 
\begin{cases} 
st & \text{if } C_{jk} = x_i, \\ 
ts & \text{if } C_{jk} = \neg x_i, \\ 
\n  &\text{otherwise,}
\end{cases} \\
d_{jk}(n+i) &:= \begin{cases} 
st & \text{if } C_{jk} = y_i, \\ 
ts & \text{if } C_{jk} = \neg y_i, \\ 
\n  &\text{otherwise,}
\end{cases} \\
d_{jk}(2n+i) &:= st. \\
\end{split}\end{align*}
For $i\in\rg m$ let
\begin{align*}\begin{split}
d_{jk}(3n+i) &:= \begin{cases} 
t &\text{if } i=j, \\
\n  &\text{otherwise,}
\end{cases} \\
d_{jk}(3n+m+1) &:= 
\rlap{$ts$.}
\hphantom{
\begin{cases} 
st & \text{if } C_{jk} = x_i, \\ 
ts & \text{if } C_{jk} = \neg x_i, \\ 
\n  &\text{otherwise,}
\end{cases}} \\
\end{split}\end{align*}

\item
Let $j\in\rg{n}$. After each assignment was succesfully evaluated, 
the tuple $e_j$ sets position $n+j$ to $st$ in order to match the target tuple if necessary.
We define
\begin{align*}\begin{aligned}
e_j(i) &:= st, &&\text{for $i\in\rg{n}$,}\\
e_j(n+i) &:= \rlap{$\begin{cases}
t  & \text{if } i=j, \\
\n & \text{if } i\in \rg{2n}\setminus \{j\}, \\
\end{cases}$} &&\\
e_j(3n+i) &:= st &&\text{for $i\in\rg{m}$,} \\
e_j(3n+m+1) &:= ts. && \\
\end{aligned}\end{align*}
The tuples $e_1,\ldots,e_n$ will only occur as final generators of $f$.
\end{itemize}

\newcommand{\vdotp}{\raisebox{-2pt}{\text{\smash{$\vdots$}}}}
\newcommand{\ddotp}{\raisebox{-2pt}{\text{\smash{$\ddots$}}}}
\newcommand{\st}{st}
\begin{figure}
\begin{gather*}
\arraycolsep=4pt
\def\arraystretch{1.15}
\begin{array}[c]{r@{}rcccccccccccccl}
a       &{} = (\hspace{-\arraycolsep}&  s   &\cdots&  s   &  s   &\cdots&  s   &  s   &\cdots&  s   &  s   &\cdots&  s   &  s   &\hspace{-\arraycolsep}) \vspace{6pt}\\ 
b_1     &{} = (\hspace{-\arraycolsep}&  t   &  s   &  s   &      &      &      &  s   &\cdots&  s   &  s   &\cdots&  s   &  ts  &\hspace{-\arraycolsep}) \\
\vdotp\ &{}   (\hspace{-\arraycolsep}&      &\ddotp&  s   &      &      &      & \st  &\ddotp&\vdotp&\vdotp&\ddotp&\vdotp&\vdotp&\hspace{-\arraycolsep}) \\
b_n     &{} = (\hspace{-\arraycolsep}&      &      &  t   &      &      &      & \st  & \st  &  s   &  s   &\cdots&  s   &  ts  &\hspace{-\arraycolsep}) \vspace{6pt}\\
c_1^+   &{} = (\hspace{-\arraycolsep}&      &      &      &  t   &      &      &  t   &      &      &      &      &      &  ts  &\hspace{-\arraycolsep}) \\
\vdotp\ &{}   (\hspace{-\arraycolsep}&      &      &      &      &\ddotp&      &      &\ddotp&      &      &      &      &\vdotp&\hspace{-\arraycolsep}) \\
c_n^+   &{} = (\hspace{-\arraycolsep}&      &      &      &      &      &  t   &      &      &  t   &      &      &      &  ts  &\hspace{-\arraycolsep}) \vspace{6pt}\\
c_1^-   &{} = (\hspace{-\arraycolsep}&      &      &      &  s   &      &      &  t   &      &      &      &      &      &  ts  &\hspace{-\arraycolsep}) \\
\vdotp\ &{}   (\hspace{-\arraycolsep}&      &      &      &      &\ddotp&      &      &\ddotp&      &      &      &      &\vdotp&\hspace{-\arraycolsep}) \\
c_n^-   &{} = (\hspace{-\arraycolsep}&      &      &      &      &      &  s   &      &      &  t   &      &      &      &  ts  &\hspace{-\arraycolsep}) \vspace{6pt}\\
c_1^0   &{} = (\hspace{-\arraycolsep}&      &      &      &      &      &      &  t   &      &      &      &      &      &  ts  &\hspace{-\arraycolsep}) \\
\vdotp\ &{}   (\hspace{-\arraycolsep}&      &      &      &      &      &      &      &\ddotp&      &      &      &      &\vdotp&\hspace{-\arraycolsep}) \\
c_n^0   &{} = (\hspace{-\arraycolsep}&      &      &      &      &      &      &      &      &  t   &      &      &      &  ts  &\hspace{-\arraycolsep}) \vspace{6pt}\\
d_{1k}  &{} = (\hspace{-\arraycolsep}& *    &\cdots&    * & *    &\cdots& *    &  st  &\cdots&  st  &  t   &      &      &  ts  &\hspace{-\arraycolsep}) \\
\vdotp\ &{}   (\hspace{-\arraycolsep}&\vdotp&\ddotp&\vdotp&\vdotp&\ddotp&\vdotp&\vdotp&\ddotp&\vdotp&      &\ddotp&      &\vdotp&\hspace{-\arraycolsep}) \\
d_{mk}  &{} = (\hspace{-\arraycolsep}& *    &\cdots&    * & *    &\cdots& *    &  st  &\cdots&  st  &      &      &  t   &  ts  &\hspace{-\arraycolsep}) \vspace{6pt}\\
e_1     &{} = (\hspace{-\arraycolsep}&  st  &\cdots&  st  &  t   &      &      &      &      &      &  st  &\cdots&  st  &  ts  &\hspace{-\arraycolsep}) \\
\vdotp\ &{}   (\hspace{-\arraycolsep}&\vdotp&\ddotp&\vdotp&      &\ddotp&      &      &      &      &\vdotp&\ddotp&\vdotp&\vdotp&\hspace{-\arraycolsep}) \\
e_n     &{} = (\hspace{-\arraycolsep}&  st  &\cdots&  st  &      &      &  t   &      &      &      &  st  &\cdots&  st  &  ts  &\hspace{-\arraycolsep}) \vspace{6pt}\\
f       &{} = (\hspace{-\arraycolsep}&  st  &\cdots&  st  &  st  &\cdots&  st  &  st  &\cdots&  st  &  st  &\cdots&  st  &  s   &\hspace{-\arraycolsep}) \vspace{-4pt}\\
               &{}                          &
\rlap{$\underbrace{\phantom{st\hspace{2\arraycolsep}{\cdots}\hspace{2\arraycolsep}st}\hspace{1.5pt}}_{n}$}\phantom{st}   &&&      
\rlap{$\underbrace{\phantom{st\hspace{2\arraycolsep}{\cdots}\hspace{2\arraycolsep}st}\hspace{1.5pt}}_{n}$}\phantom{st}   &&&      
\rlap{$\underbrace{\phantom{st\hspace{2\arraycolsep}{\cdots}\hspace{2\arraycolsep}st}\hspace{1.5pt}}_{n}$}\phantom{st}   &&&      
\rlap{$\underbrace{\phantom{st\hspace{2\arraycolsep}{\cdots}\hspace{2\arraycolsep}st}\hspace{1.5pt}}_{m}$}\phantom{st}   &&&
\end{array}
\end{gather*}
\caption{Generators and target tuple of the $\smp(T)$ instance for $k\in\rg{3}$. Empty positions encode the element $\n$, and the symbol ``$*$'' indicates that this entry depends on the \dprobfont{Q3SAT} instance.}
\label{fig:gens}
\end{figure}

Now we state what we already mentioned in the definition of $d_{jk}$.
\begin{clm}\label{cl1_mthm}
Let $h \in \subuni G$ such that
\begin{align*}\begin{aligned}
h(i) &\in \{s,st\} &&\text{for all } i\in\rg{2n}, \\
h(2n+i) &=  st &&\text{for all } i\in\rg{n}, \\
h(3n+j) &=  s &&\text{for some } j\in\rg{m}, \\
h(3n+m+1) &=  s. &&\\
\end{aligned}\end{align*}
Let $\rho\colon \{ x_1,\ldots,x_n,y_1,\ldots,y_n \} \to \{\0,\1\}$ be the assignment encoded by $h|_{\rg{2n}}$, i.e.
\begin{align*}\begin{aligned}
\rho(x_i) &:= h(i)   & \text{for } i\in\rg{n},\\
\rho(y_i) &:= h(n+i) & \text{for } i\in\rg{n}.
\end{aligned}\end{align*}
Let $k\in\rg{3}$, $C_{jk}$ be the $k$th literal of the $j$th clause, and $h':=hd_{jk}$.
\begin{enumerate}
\item\label{it1_cl1_mthm} 
If $\rho(C_{jk})=\1$, then
$h$ and $h'$ differ only in the following position: \[ h'(3n+j) = st. \]

\item\label{it2_cl1_mthm} 
Otherwise $h'(i)\gjless s$ for some $i\in\rg{2n}$.
\end{enumerate}
\end{clm}

\begin{proof}
The literal $C_{jk}$ is of the form $z$ or $\neg z$ for some variable $z$. 
Let $\ell\in\rg{2n}$ be the position of $z$ in $(x_1,\ldots,x_n,y_1,\ldots,y_n)$.
First assume $C_{jk}$ is of the form $z$. Then $d_{jk}(\ell) = st$.

\lref{it1_cl1_mthm} Assume  $\rho(C_{jk})=\1$. This means $h(\ell) = st$. Thus $h'(\ell) = st \cdot st = st = h(\ell)$, and $h'(3n+j) = s \cdot t = st$. 
It is easy to see that $h(i)=h'(i)$ for the remaining positions $i\in\rg{3n+m+1}\setminus \{\ell,3n+j\}$.

\lref{it2_cl1_mthm} Assume  $\rho(C_{jk})=\0$. 
This means  $h(\ell) = s$. Thus $h'(\ell) = s \cdot st \gjless s$, and \lref{it2_cl1_mthm} is proved.

If $C_{jk}$ is of the form $\neg z$, then~\lref{it1_cl1_mthm} and~\lref{it2_cl1_mthm} are proved in a similar way.
\end{proof}

Note that if $\rho$ satisfies $C_{jk}$, then multiplying $h$ by $d_{jk}$ changes the status of the $j$th clause from ``unsatisfied'' to ``satisfied''. 
Otherwise the target tuple $f$ cannot be reached by further multiplying $hd_{jk}$ with elements of $G$.

In the remainder of the proof we show the following.
\begin{clm}\label{cl2_mthm}
$\Phi$ holds if and only if $f\in\subuni G$.
\end{clm}

\emph{$(\Rightarrow)$ direction of Claim~\ref{cl2_mthm}.}
Assume $\Phi$ is true. This means that for every $i\in \rg{n}$ there is a function
$\psi_i:\{\0,\1\}^i\to\{\0,\1\}$  such that for every assignment $\varphi:\{x_1,\dots,x_n\}\to\{\0,\1\}$
the assignment
\[ \rho_\varphi:=\varphi\cup\{y_i \mapsto \psi_i(\varphi(x_1),\dots,\varphi(x_i))\mid i\in[n]\} \]
satisfies all the clauses $C_1,\dots, C_m$. 

We prove by induction on assignments $\varphi$ in lexicographic order that for each $\varphi$ the following tuple $g_\varphi$ belongs to $\subuni G$:
\begin{align*}\begin{aligned}\
g_\varphi(i) &:= \varphi(x_i) && \text{for } i\in\rg{n}, \\ 
g_\varphi(n+i) &:= \rho_\varphi(y_i) && \text{for } i\in\rg{n}, \\
g_\varphi(2n+i) &:= st &&\text{for } i\in\rg{n+m}, \\
g_\varphi(3n+m+1) &:= s. &&
\end{aligned}\end{align*}
For the base case let $\varphi(x_i):=\0$ for all $i\in\rg{n}$. 
For $i\in\rg{n}$ we define
\[ c'_i := \begin{cases} 
c_i^+ & \text{if }  \rho_\varphi(y_i) = \1, \\ 
c_i^0 & \text{otherwise.} \end{cases} \]
Apparently $a\cdot c_1'\cdots c_n'|_{\rg{2n}}$ encodes $\rho_{\varphi}$.
For each $j \in \rg{m}$ there is a $k_j \in \rg{3}$ such that the literal $C_{jk_j}$ is satisfied by $\rho_{\varphi}$.
By Claim~\ref{cl1_mthm}\tlref{it1_cl1_mthm} it is straightforward to verify that
\begin{equation*}\label{eq10_mthm}
g_\varphi = a \cdot c'_1\cdots c'_n \cdot d_{1k_1} \cdots d_{mk_m}. 
\end{equation*}

Now let $\varphi$ be an assignment with successor $\varphi'$ in lexicographical order such that $g_\varphi\in \subuni G$. 
Let $j \in \rg{n}$ be maximal such that $\varphi(x_j)=\0$. Then
\begin{align*}\begin{aligned}
 \varphi(x_i) & =\varphi'(x_i) &&\text{for } i<j, \\
 \varphi(x_j) &=\0,\ \varphi'(x_j)=\1, &&\\ 
 \varphi(x_i) &=\1,\ \varphi'(x_i)=\0 &&\text{for } j<i\leq n.
\end{aligned}\end{align*}
To adjust the assignment for the existential variables, for $j\le i\le n$ set
\[  c'_i := \begin{cases}  
c^{+}_i & \text{if } \rho_\varphi(y_i)=\0,\  \rho_{\varphi'}(y_i)=\1, \\
c^{-}_i & \text{if } \rho_\varphi(y_i)=\1,\  \rho_{\varphi'}(y_i)=\0,  \\                       
c^0_i & \text{otherwise.}
\end{cases} \]
For $h:=g_\varphi \cdot b_j \cdot c'_j\cdots c'_n$ we have
\begin{align*}\begin{aligned}
h(i) & =\varphi'(x_i)           &&\text{for } i\in\rg{n}, \\
h(n+i) & = \rho_{\varphi'}(y_i) &&\text{for } i\in\rg{n}, \\ 
h(2n+i) & = st &&\text{for } i\in\rg{n}, \\
h(3n+i) & = s &&\text{for } i\in\rg{m+1}. \\
\end{aligned}\end{align*}
For each $j \in \rg{m}$ the clause $C_j$ is satisfied by $\rho_{\varphi'}$. 
Thus there is a $k_j \in \rg{3}$ such that $\rho_{\varphi'}$ satisfies the literal $C_{jk_j}$.
From Claim~\ref{cl1_mthm}\tlref{it1_cl1_mthm} follows that
\[ g_{\varphi'} = h \cdot d_{1k_1} \cdots d_{mk_m}. \]
This completes the induction argument.

Finally let $\varphi$ be such that $\varphi(x_i)=\1$ for all $i\in\rg{n}$, 
and $g_\varphi$ as defined above.
Denote the positions $i\in\rg{n}$ where $g_\varphi(n+i)=s$ by $i_1,\ldots,i_p$.
Then we have \[f = g_\varphi \cdot e_{i_1} \cdots e_{i_p}.\]
Thus $f\in\subuni G$.
The $(\Rightarrow)$ direction of Claim~\ref{cl2_mthm} is proved.

We give another description of the product that yields $f$. 
For each assignment $\varphi\neq\0$ for $x_1,\dots,x_n$ let 
$j_\varphi := \max\{ j\in\rg{n} \mid \varphi(x_j) = \1 \}$. 
From our argument above we see that $f$ is of the form
\begin{equation} \label{eq20_mthm}
f = a c_1^*\cdots c_n^* d_{1\dagger} \cdots d_{m\dagger} \cdot 
(\prod_{\mathclap{\varphi \neq \0}} b_{j_\varphi} c_{j_\varphi}^* \cdots c_n^* d_{1\dagger} \cdots d_{m\dagger}) 
\cdot e_{i_1} \cdots e_{i_p},
\end{equation}
where $i_1,\ldots,i_p\in\rg{n}$ are distinct, each $*$ belongs to $\{+,-,0\}$, and each $\dagger$ to $\rg{3}$.
The product is taken over all assignments $\varphi \neq \0$ to $x_1,\dots,x_n$ in
lexicographical order. 

\emph{$(\Leftarrow)$ direction of Claim~\ref{cl2_mthm}.}
Assume $ f\in\subuni G$. 
Let $k\in\mathbb{N}$ be minimal such that $f=u_1\cdots u_k$ for some $u_1,\dots, u_k \in G$, 
and let $ v_i:= u_1\cdots u_i$ for $i\in\rg{k}$. 

\begin{clm}\label{cl3_mthm}
Let $i\in\{2,\ldots,k\}$ and $j\in\rg{3n+m+1}$. Then
\begin{enumerate}
\item\label{it1_cl3_mthm}
$ u_1 =  a$ and $ u_i\neq  a$,
\item\label{it2_cl3_mthm}
 $ v_i(j) \in \{s,st\}$.
\item\label{it3_cl3_mthm}
If $u_i(j) \in \{\n,st,ts\}$, then $v_i(j) = v_{i-1}(j)$.
\end{enumerate}
\end{clm}

\begin{proof}
\lref{it1_cl3_mthm}
If $u_i=a$, then $u_{i-1}u_i(3n+m+1) \in \{ts^2, s^2\}$,
which yields the contadiction $v_k(3n+m+1) \gjless s$.
If $u_1 \neq a$, then we obtain the contradiction $v_k(3n+m+1)=ts$.

\lref{it2_cl3_mthm}
We use induction on $i$. By~\lref{it1_cl3_mthm} we have $v_1=a$. Thus $v_1(j) \in \{s,st\}$.
Now assume $v_{i-1}(j)\in\{s,st\}$. 
Since $u_i(j)\in\{\n,s,st,t,ts\}$,
either $v_i(j)\in\{s,st\}$, or a factor $s^2$ or $t^2$ occurs in $v_i(j)$. 
As $s^2,t^2\gjless f(j)$, 
the first case applies.

\lref{it3_cl3_mthm}
is immediate from item~\lref{it2_cl3_mthm} and the multiplication table~\eqref{eq05_mthm}.
\end{proof}

The next claim states that the product $u_1\cdots u_k$ is of a similar form to the one given in~\eqref{eq20_mthm}.

\begin{clm}\label{cl4_mthm}
Let $i\in\rg{k}$. Let either $j\in\rg{n}$ and $u_i=b_j$, or $j=1$ and $u_i=a$. 
Then for $i_1:=i+n-j+1$ and $i_2:=i_1+m$ the following holds:

\begin{enumerate}
\item\label{it1_cl4_mthm}
$\{u_{i+1},\ldots,u_{i_1}\} = \{c_{j}^{p_j},\ldots,c_{n}^{p_n}\}$ for some $p_1,\ldots,p_m \in \{+,-,0\}$.

\item\label{it2_cl4_mthm}
$\{u_{i_1+1},\ldots,u_{i_2}\} = \{d_{1k_1},\ldots,d_{mk_m}\}$ for some $k_1,\ldots,k_m \in \rg{3}$.

\item\label{it3_cl4_mthm}
If there is a greatest $j'\in\rg{n}$ such that $v_i(j')=s$, then $u_{i_2+1} =b_{j'}$;

\item\label{it4_cl4_mthm}
otherwise $u_{i_2+1},\ldots,u_k$ are distinct and form the set \\$\{ e_\ell \mid \ell \in \rg{n},\ v_{i_2}(n+\ell)=s \}$.
\end{enumerate}
\end{clm}
\begin{proof}
First let $j\in\rg{n}$ and $u_i=b_j$.

\lref{it1_cl4_mthm}
From Claim~\ref{cl3_mthm} we know that $i\ge 2$ and that every coordinate of $v_{i-1}$ is either $s$ or $st$. From the definition of $b_j$ and the multiplication table~\eqref{eq05_mthm} we know that
\[v_i(2n+\ell)= \begin{cases} st &\text{for } \ell \in \rg{j-1}, \\ s &\text{for } \ell\in\{j,\ldots,n\}.\end{cases}\]
Thus the only choice for the $n-j+1$ generators subsequent to $u_i$ is given by $c_{j}^{p_j},\ldots,c_{n}^{p_n}$ for some $p_j,\ldots,p_n \in \{+,-,0\}$, where the order does not matter; otherwise we would obtain a factor $s^2$ or $t^2$ in a position $2n+\ell$ for some $\ell\in\rg{n+m}$, which is impossible.

\lref{it2_cl4_mthm}
From~\lref{it1_cl4_mthm} we know that 
\begin{align*}\begin{aligned}
v_{i_1}(2n+\ell)&=st & \text{for } &\ell\in\rg{n}, \\
v_{i_1}(3n+\ell)&=s  & \text{for } &\ell\in\rg{m}.
\end{aligned}\end{align*}
Thus the $m$ generators subsequent to $u_{i_1}$ are given by $d_{1k_1},\ldots,d_{mk_m}$ for some $k_1,\ldots,k_m \in \rg{3}$ where the order does not matter;
otherwise we would obtain a factor $t^2$ in $v_k(2n+\ell)$ for some $\ell\in\rg{n}$, or $s^2$ in $v_k(3n+\ell)$ for some $\ell \in \rg{m}$.
Both cases contradict the fact that $v_k=f$.

\lref{it3_cl4_mthm}
Let $j'\in\rg{n}$ be maximal such that $v_i(j')=s$.
From~\lref{it1_cl4_mthm} and~\lref{it2_cl4_mthm} we know that 
\begin{align*}\begin{aligned}
	v_{i_2}(j')&=s &&\text{and} \\
   v_{i_2}(2n+\ell)&=st  && \text{for } \ell\in\rg{n+m}. \\
\end{aligned}\end{align*}
If $u_{i_2+1} \in E$, then $v_{i_2+1}(j')= s^2t$.
If $u_{i_2+1} \in C\cup D$, then $v_{i_2+1}(2n+\ell)= st^2$ for some $\ell\in\rg{n+m}$.
Thus $u_{i_2+1}=b_\ell$ for some $\ell\in\rg{n}$.
If $\ell<j'$, then $v_{i_2+1}(j')=s^2$.
If $\ell>j'$, then $v_{i_2+1}(\ell)= st^2$.
Therefore $\ell=j'$.

\lref{it4_cl4_mthm}
Assume $v_i(\ell)=st$ for all $\ell\in\rg{n}$. 
Suppose some generator among $u_{i_2+1},\ldots,u_k$ belongs to $B$.
Let $i_3\in\{{i_2+1},\ldots,k\}$ be minimal such that $u_{i_3} \in B$.
By~\lref{it1_cl4_mthm}, \lref{it2_cl4_mthm}, and 
Claim~\ref{cl3_mthm}\tlref{it3_cl3_mthm} 
we have \[ v_i|_{\rg{n}}=\ldots=v_{i_3-1}|_{\rg{n}}. \]
Thus $v_{i_3}(\ell)=st^2$ for some $\ell\in\rg{n}$, which is impossible.
Hence $u_{i_2+1},\ldots,u_k \notin B$.
This together with~\lref{it1_cl4_mthm} and~\lref{it2_cl4_mthm} implies
\[\text{$v_{i_2}(2n+\ell)=\ldots=v_k(2n+\ell)=st$\quad{}for all $\ell\in\rg{n+m}$.}\]
Thus $u_{i_2+1},\ldots,u_k \notin C\cup D$. Otherwise we would have a factor $t^2$ in $v_k(n+\ell)$.
So $u_{i_2+1},\ldots,u_k \in E$.
If $u_{i_2+1},\ldots,u_k$ were not distinct, then we had a factor $t^2$ in $v_k(n+\ell)$ for some $\ell\in\rg{n}$. Finally observe that for each $\ell\in\rg{n}$ with $v_{i_2}(n+\ell)=s$ we have $e_\ell\in \{u_{i_2+1},\ldots,u_k\}$; otherwise $v_k(n+\ell)=s$ which is impossible. We proved \lref{it4_cl4_mthm}.

For $j=1$ and $u_i=a$, items \lref{it1_cl4_mthm}~to~\lref{it4_cl4_mthm} are proved in a similar manner.
\end{proof}

In the following we define assignments to the variables using the first $2n$ positions of the tuples $v_1,\ldots,v_k$.
For $i\in\rg{k}$ and $j\in\rg{n}$ let 
\begin{align*}\begin{aligned}
\varphi_i \colon \{x_1,\dots,x_n\} &\to \{\0,\1\},&
\varphi_i(x_j) &:= v_i(j), \\
\theta_i \colon \{y_1,\dots,y_n\} &\to\{\0,\1\},&
\theta_i(y_j) &:= v_i(n+j). \\
\end{aligned}\end{align*}
These assignments fulfill the following conditions.

\begin{clm}\label{cl5_mthm} \mbox{}
\begin{enumerate}

\item\label{it0_cl5_mthm} 
For $i\in\rg{k-1}$ we have $\varphi_i\ne\varphi_{i+1}$ if and only if $u_{i+1}\in B$.

\item \label{it1_cl5_mthm}
$\varphi_1,\dots,\varphi_{k}$ is a list of all assignments for $x_1,\dots,x_n$ (possibly with repetitions) in lexicographic order.

\item \label{it2_cl5_mthm}
Let $i\in\rg{k-1}$ such that $u_i\in D$ and $u_{i+1}\notin D$.
Then $\varphi_{i}\cup\theta_{i}$ satisfies all the clauses $C_1,\dots,C_m$.
\end{enumerate}
\end{clm}

\begin{proof}
\lref{it0_cl5_mthm} follows from the definitions of the generators, Claim~\ref{cl3_mthm}\tlref{it3_cl3_mthm}, and Claim~\ref{cl4_mthm}.

\lref{it1_cl5_mthm} 
Let $i\in\rg{k-1}$ such that $\varphi_i \ne \varphi_{i+1}$.
By \lref{it0_cl5_mthm}  and Claim~\ref{cl4_mthm} $u_{i+1}=b_j$ for the greatest $j\in\rg{n}$ for which $\varphi_i(x_j)=\0$. 
It is easy to see that $\varphi_{i+1}(x_\ell)=\varphi_i(x_\ell)$ for $\ell<j$, 
$\varphi_{i+1}(x_j)=\1$, and $\varphi_{i+1}(x_\ell)=\0$ for $\ell>j$. 
Thus $\varphi_{i+1}$ is the successor of $\varphi_i$ in lexicographic order.
By Claim~\ref{cl3_mthm}\tlref{it1_cl3_mthm} $\varphi_1$ is the all-zero assignment for $x_1,\ldots,x_n$. 
Since $v_k(i)=st$ for all $i\in\rg{n}$, $\varphi_k$ is the all-one assignment.
Hence $\varphi_1,\dots,\varphi_{k}$ is a list of all assignments.

\lref{it2_cl5_mthm}
By Claim~\ref{cl4_mthm} $\{u_{i-m+1},\ldots,u_{i}\} = \{d_{1k_1},\ldots,d_{mk_m}\}$ for some 
$k_1,\ldots,k_m \in \rg{3}$. 
Thus $\varphi_{i-m}\cup\theta_{i-m}=\ldots=\varphi_{i}\cup\theta_{i}$.
Suppose $\varphi_{i}\cup\theta_{i}$ does not satisfy some clause $C_j$ for $j\in\rg{m}$.
Then its $k$th literal $C_{jk_j}$ is also unsatisfied.
Claim~\ref{cl1_mthm}\tlref{it2_cl1_mthm} implies $v_i(\ell)\gjless s$ for some $\ell\in\rg{2n}$, which is impossible.
\end{proof}

For $\varphi\colon\{x_1,\dots,x_n\}\to\{\0,\1\}$ let 
\[i_\varphi := \max\{ i\in\rg{k} \mid \varphi_i = \varphi,\, u_i\in D \}.\]
From Claim~\ref{cl5_mthm} we know that for every assignment $\varphi$ to $x_1,\ldots,x_n$ the assignment $\varphi\cup\theta_{i_\varphi}$ satisfies all the clauses of $\Phi$.
It only remains to prove that $\theta_{i_\varphi}(y_i)$ only depends on $\varphi(x_1),\ldots,\varphi(x_i)$.

\begin{clm}\label{cl6_mthm}
Let $i\in\rg{n}$. For all $\varphi,\chi\colon\{x_1,\dots,x_n\}\to\{\0,\1\}$ the equations
\begin{equation}\label{eq1_cl6_mthm}
\varphi(x_1) = \chi(x_1),\ \ldots,\ \varphi(x_i) = \chi(x_i) 
\end{equation}
imply $\theta_{i_\varphi}(y_i) =  \theta_{i_\chi}(y_i)$.
\end{clm}

\begin{proof}
We consider $\varphi$ as fixed and prove the implication for all $\chi\ge\varphi$ 
by induction in lexicographical order.
The base case $\chi=\varphi$ is clear.

Now let $\chi\ge\varphi$ be an assignment for which the implication holds,
and assume its successor $\chi'$ fulfills~\eqref{eq1_cl6_mthm}.
Since $\varphi \le \chi < \chi'$, the assignment $\chi$ also fulfills~\eqref{eq1_cl6_mthm}.
From Claim~\ref{cl4_mthm} we know that
\begin{align}\label{eq2_cl6_mthm}
\{u_{i_{\chi}+1},\ldots,u_{i_{\chi'}}\}=\{b_j,c_j^{p_j},\ldots,c_n^{p_n},d_{1k_1},\ldots,d_{mk_m}\}
\end{align} 
for some $j\in\rg{n}$, $p_j,\ldots,p_n\in\{+,-,0\}$, and $k_1,\ldots,k_m\in\rg{3}$.
If $j\le i$ was true, then $\chi'$ would not fulfill $\eqref{eq1_cl6_mthm}$.
Thus $j>i$. From~\eqref{eq2_cl6_mthm} follows $v_{i_{\chi'}}(n+i)=v_{i_\chi}(n+i)$.
Thus $\theta_{i_{\chi'}}(y_i)=\theta_{i_\chi}(y_i)=\theta_{i_\varphi}(y_i)$, and Claim~\ref{cl6_mthm} is proved.
\end{proof}

We complete the proof of Claim~\ref{cl2_mthm}. By Claims~\ref{cl5_mthm} and~\ref{cl6_mthm}, 
for each assignment $\varphi$ for the universal variables there is an assignment 
$\theta_{i_\varphi}$ for the existential variables such that $\varphi\cup\theta_{i_\varphi}$ 
satisfies the conjunctive normal form in $\Phi$. For all $i\in\rg{n}$ the value $\theta_{i_\varphi}(y_i)$ depends only on $\varphi(x_1),\dots,\varphi(x_i)$.
Thus $\Phi$ is true. Claim~\ref{cl2_mthm} and Lemma~\ref{lma:mthm_pspace} are proved.
\end{proof}

\begin{proof}[Proof of Theorem~\ref{thm:mthm_pspace}]
Let $s':=(s,tst)$, $t':=(tst,s)$, and $\n':=(\n,\n)$ be elements of $S^2 := S \times S$. 
Apparently $s't's'=s'$ and $t's't'=t'$.
Both $s'$ and $t'$ do not generate groups.
By Lemma~\ref{lma:j_h_classes_1} $s'^2\gjless s'$ and $t'^2\gjless t'$.
Since $t' \gj s'$, we have $t'^2\gjless s'$.
Now $s',t',\n'$ fulfill the hypothesis of Lemma~\ref{lma:mthm_pspace}.
Thus $\smp(S^2)$ is \pspace-complete.
As $\smp(S^2)$ reduces to $\smp(S)$ and conversely, the result follows.
\end{proof}

Now we are able to list several ``naturally occuring'' semigroups with \pspace-complete \smp: 
\begin{clr}\label{clr1_mthm}
The \smp{} for the following semigroups is \pspace-complete:
\begin{enumerate}
\item\label{it1_clr1_mthm}
the \emph{Brandt monoid}  $B_2^1$ and the monoid $A_2^1$;
\item\label{it2.5_clr1_mthm}
for $n\ge 2$ and a finite ring $R$ with identity $1\neq0$, the semigroup of all $n \times n$ matrices over $R$;
\item\label{it3_clr1_mthm}
the full transformation semigroup $\Tsg_n$ on $n\ge 3$ letters;
\item\label{it4_clr1_mthm}
the symmetric inverse semigroup $\Isg_n$ on $n\ge 2$ letters.
\end{enumerate}
\end{clr}

\begin{proof}
We apply Theorem~\ref{thm:mthm_pspace}.

\lref{it1_clr1_mthm} 
For $B_2^1$ let $s := [1,2]$ and $t := [2,1]$.
For $A_2^1$ let $s := [2,2]$ and $t := [1,1]$.

\lref{it2.5_clr1_mthm} 
Define $n \times n$ matrices $s,t$ over $R$ by
\begin{align*} 
s_{ij} := \begin{cases} 1 &\text{if } (i,j)=(1,2), \\ 0 &\text{otherwise,} \end{cases} \qquad
t_{ij} := \begin{cases} 1 &\text{if } (i,j)=(2,1), \\ 0 &\text{otherwise} \end{cases}
\end{align*}
for $i,j\in\rg{n}$. Let $\n$ be the identity matrix.

\lref{it3_clr1_mthm} 
Let $\n$ be the identity mapping on $\rg{n}$, and $s,t \colon \rg{n}\rightarrow\rg{n}$, 
\begin{align*}
s(x) &:= \begin{cases} 2 &\text{if } x=1, \\ 3&\text{otherwise,} \end{cases} \qquad
t(x) := \begin{cases} 1 &\text{if } x=2, \\ 3&\text{otherwise.} \end{cases}
\end{align*}

\lref{it4_clr1_mthm}
Let $\n$ be the identity mapping, $s\colon 1 \mapsto 2$, and $t\colon 2 \mapsto 1$.
\end{proof}

For monoids we can now generalize Corollary~\ref{clr:hardness_jclass_argument}:
\begin{clr}
If a $\gj$-class of a finite monoid $S$ contains both group and non-group $\gh$-classes, then $\smp(S)$ is \pspace-complete.
\end{clr}

\begin{proof}
Let $S$ be as above.
Similar to the proof of Corollary~\ref{clr:hardness_jclass_argument}, there is a $t\in S$ such that $sts=s$.
Now $s$, $t$, and the identity fulfill the hypothesis of Theorem~\ref{thm:mthm_pspace}.
\end{proof}

\section{Proof of Theorem~\ref{thm:trichotomy_comb_rmsg_id}}
\label{sec:rmsg_id}

\begin{lma}\label{lma:crms_id_in_np}
If the $0$-$1$ matrix $P$ of a finite combinatorial Rees matrix semigroup $S_P$ has one block, then $\smp(S_P^1)$ is in \np.
\end{lma}

\begin{proof}
Assume $P\in\{0,1\}^{\Lambda \times I}$, and let $J\subseteq I$ and $\Delta \subseteq \Lambda$ 
such that $P(\lambda,i)=1$ if and only if $(\lambda,i) \in \Delta \times J$ for $i\in I$, $\lambda \in \Lambda$.
Let $T:=S_P^1$ and $A \subseteq T^n,\, b\in T^n$ be an instance of $\smp(T)$ such that $b\in\subuni{A}$.
Let $a_1,\ldots,a_k \in A$ such that $b=a_1\cdots a_k$. 
If $b=(1,\ldots,1)$ or $k=1$, then clearly $b\in A$. In this case the position of $b$ in the list $A$ is a witness.
Assume $b\ne (1,\ldots,1)$ and $k\ge2$.

We claim that for $i\in\rg{n}$ with $b(i)=0$ 
there are $\ell_i,r_i \in\rg{k}$, $\ell_i < r_i$ such that 
\begin{equation}\label{eq1_lma:crmm_oneblock_in_np}
a_{\ell_i} a_{r_i}(i) = 0 \quad\text{and}\quad a_{\ell_i+1}(i)=\ldots=a_{r_i-1}(i) = 1. 
\end{equation}
This follows from Lemma~\ref{lma:combrmsg1}\tlref{it1_lma:combrmsg1}. 
For $i\in\rg{n}$ with $b(i)\in I \times \Lambda$ let
\begin{align*} 
\ell_i := \min&\{ j\in\rg{k} \sst a_j(i) \ne 1 \}, \\
r_i := \max&\{ j\in\rg{k} \sst a_j(i) \ne 1 \}.
\end{align*}
Now define an index set $N \subseteq \rg{k}$ by
\[ N := \{ \ell_i \sst i\in\rg{n},\,b(i)\ne 1 \} \cup \{ r_i \sst i\in\rg{n},\,b(i)\ne 1 \}. \]
Note that $N\neq\varnothing$; otherwise $b=(1,\ldots,1)$ which contradicts our assumption. 

For $i\in\rg{n}$ we claim that 
\begin{equation}\label{eq2_lma:crmm_oneblock_in_np} 
\prod_{j\in N}a_j(i) = b(i),
\end{equation}
where the indexes $j$ of the factors are in ascending order.
If $b(i)=1$, then $a_j(i)=1$ for all $j\in\rg{k}$, and \eqref{eq2_lma:crmm_oneblock_in_np} follows.
Assume $b(i)=0$. We have $\ell_i,r_i \in N$.
By \eqref{eq1_lma:crmm_oneblock_in_np} all factors in \eqref{eq2_lma:crmm_oneblock_in_np} between $a_{\ell_i}(i)$ and $a_{r_i}(i)$  are equal to $1$. 
This and \eqref{eq1_lma:crmm_oneblock_in_np} imply \eqref{eq2_lma:crmm_oneblock_in_np}.
Finally assume $b(i)\in I\times\Lambda$. 
For $\ell_i < j < r_i$ we have $a_j(i)\in \{1\}\cup(J\times\Delta)$; otherwise we obtain the contradiction $b(i)=0$.
Thus
\begin{align*}\begin{aligned}
\prod_{\mathclap{\substack{ j\in N \\ \ell_i \le j \le r_i }}} a_j(i) 
= \prod_{\mathclap{\substack{ \ell_i \le j \le r_i }}} a_j(i).
\end{aligned}\end{align*}
Since $a_j(i)=1$ for $j<\ell_i$ and $j>r_i$, \eqref{eq2_lma:crmm_oneblock_in_np} follows.

The length of the product in \eqref{eq2_lma:crmm_oneblock_in_np} is $|N|$ and thus at most $2n$.
Thus this product is a valid witness for $b \in\subuni{A}$, and the lemma is proved.
\end{proof}

\begin{proof}[Proof of Theorem~\ref{thm:trichotomy_comb_rmsg_id}]
Assume $P\in\{0,1\}^{\Lambda \times I}$.

\lref{it1_thm:trichotomy_comb_rmsg_id} 
If $P$ is the all-$1$ matrix, then $S_P^1$ is a band (idempotent semigroup) 
with $\gj$-classes $\{0\}$, $I \times \Lambda$, and $\{1\}$.
We show that $S_P^1$ is a \index{regular!band}\emph{regular band}, that is, $S_P^1$ satisfies the identity
\begin{equation}\label{eq1_thm:trichotomy_comb_rmsg_id}
xyxzx=xyzx.
\end{equation}
Let $x,y,z \in S_P^1$. 
If one of the variables is $0$ or $1$, then \eqref{eq1_thm:trichotomy_comb_rmsg_id} clearly holds.
If $x,y,z \in I\times\Lambda$, then $xyxzx = x = xyzx$ by the definition of the multiplication.
Thus $S_P^1$ is a regular band. By \cite[Corollary 1.7]{smpbands} the \smp{} for every regular band is in \ptime.

\lref{it2_thm:trichotomy_comb_rmsg_id}
Assume $P$ has one block and some entries are $0$.
Let $i\in I$ and $\lambda\in\Lambda$ such that $P(\lambda,i)=0$. Let $s:=[i,\lambda]$ and $r:=t:=1$.
Since $s$ does not generate a group, $\smp(S_P^1)$ is \np-hard by Theorem~\ref{thm:nphard_condition}.
\np-easiness follows from Lemma~\ref{lma:crms_id_in_np}.

\lref{it3_thm:trichotomy_comb_rmsg_id}
In this case $P$ does not have one block.
Thus there are $i,j \in I$ and $\lambda,\mu \in \Lambda$ such that
\[ P(\lambda,i) = P(\mu,j) = 1 \quad\text{and}\quad P(\lambda,j)=0. \]
Let $s:=[j,\lambda]$ and $t:=[i,\mu]$.
Then $s$ does not generate a group, $sts=s$, $s1=s$, and $t1=t$.
By Theorem~\ref{thm:mthm_pspace} $\smp(S)$ is \pspace-complete.
\end{proof}

\section{Conclusion}
\label{sec:conclusion}

In Section~\ref{sec:comb_reesmatrix} we established a \ptime/\np-complete dichotomy for combinatorial Rees matrix semigroups. 
The next goal is to investigate the complexity for the more general case of Rees matrix semigroups.
For Rees matrix semigroups without $0$ a polynomial time algorithm for the \smp{} is known \cite{SteindlPhD2015}.
However, the following questions are open:

\begin{prb}
Is the \smp{} for finite Rees matrix semigroups (with $0$) in \np?
In particular, is there a \ptime/\np-complete dichotomy?
\end{prb}

In Section~\ref{sec:pspace} we saw the first example of a semigroup with \np-complete \smp{} where adjoining an identity results in a \pspace-complete \smp{}. 
This leads to the following question:

\begin{prb}
How hard is the \smp{} for finite Rees matrix semigroups with adjoined identity?
\end{prb}

\noindent
The answer is not even known for 
the completely regular case.
E.g.\ the complexity for the following 9-element semigroup is open:
\begin{prb}
Let $1,c$ be the elements of the cyclic group $\mathbb{Z}_2$ such that $c^2=1$,
and define a Rees matrix semigroup $S := \rmsg{\mathbb{Z}_2}{\rg{2}}{\rg{2}}
{\left(\begin{smallmatrix} 1 & 1 \\ 1 & c \end{smallmatrix}\right)}$.
How hard is $\smp(S^1)$?
\end{prb}

\bibliographystyle{abbrv} 

\end{document}